\documentclass[12pt]{amsart}
\usepackage{amscd,amssymb,epsfig}
\usepackage{graphics}

\input epsf
\newtheorem{theorem}{Theorem}[section]
\newtheorem{lemma}[theorem]{Lemma}
\newtheorem{proposition}[theorem]{Proposition}
\newtheorem{corollary}[theorem]{Corollary}
\theoremstyle{definition}

\newtheorem{example}[theorem]{Example}

\newtheorem{question}[theorem]{Question}
\newtheorem{remark}[theorem]{Remark}
\newtheorem{fact}[theorem]{Fact}
\newtheorem{claim}[theorem]{Claim}
\numberwithin{figure}{section}
\numberwithin{table}{section}
\newcommand{\ZZ}{\mathbb{Z}}
\newcommand{\Aut}{\operatorname{Aut}}
\newcommand{\Ker}{\operatorname{Ker}}
\newcommand{\id}{\operatorname{id}}

\usepackage[dvipsnames]{xcolor}

\begin{document}

\title[Marked fatgraph complexes and surface automorphisms]
{Marked fatgraph complexes and surface automorphisms}
\
\author{Yusuke Kuno}
\address{Department of Mathematics\\
Graduate School of Science\\
Hiroshima University\\
1-3-1 Higashi-Hiroshima, Hiroshima 739-8526\\
JAPAN}
\email{kunotti{\char'100}hiroshima-u.ac.jp}
\
\author{R.\ C.\ Penner}
\address{Center for the Quantum Geometry of Moduli Spaces\\
Aarhus University\\
DK 8000 C Aarhus\\
Denmark~~{\rm and}~~Departments of Mathematics and Theoretical Physics\\
Caltech, Pasadena CA 91125 USA}
\email{rpenner{\char'100}imf.au.dk~~{\rm and}~~rpenner{\char'100}caltech.edu}
\
\author{Vladimir Turaev}
\address{Department of Mathematics\\
Indiana University\\
Bloomington IN 47405\\
USA}
\email{vtouraev{\char'100}indiana.edu}

\thanks{This work was done at the Center for Quantum Geometry of Moduli Spaces
funded by the Danish National Research Foundation.}

\keywords{Mapping class group, Torelli group, fatgraph, Johnson
homomorphism}

\begin{abstract}
Combinatorial aspects of the Torelli-Johnson-Morita theory of surface automorphisms
are extended to certain subgroups of the mapping class groups.
These subgroups are defined relative to a specified homomorphism
from the fundamental group of the surface onto an arbitrary group
$K$.  For $K$ abelian, there is a
combinatorial theory akin to the classical case, for example,
providing an explicit cocycle representing the first Johnson
homomophism with target $\Lambda ^3 K$.
 Furthermore, the Earle class with coefficients in $K$  is represented by an explicit
 cocyle.
\end{abstract}

\maketitle

\section*{Introduction}

Consider a   compact oriented surface  $\Sigma$ with non-empty
boundary and  basepoint $\ast\in \partial \Sigma$. In this note, we
study certain subgroups of the mapping class group $MC(\Sigma )$.
These subgroups are determined by an arbitrary group $K$ and a fixed
epimorphism $p $ from the fundamental group $\pi=\pi_1(\Sigma,
\ast)$ onto $K$. There are two flavors to these subgroups defined by
commutative diagrams
$$\begin{array}{cccccc}
&\pi&&\xrightarrow{~~~\varphi~~~}&&\pi\\\\
&&p~\searrow&&\swarrow~p\\\\
&&&K\\
 \end{array} \qquad {\rm and}\quad \begin{array}{cccccc}
&& \pi&\xrightarrow{~~~~\varphi~~~~}&\pi\\\\
&&p~\downarrow&&\downarrow~p\\\\
&&K&\xrightarrow{~~~\tilde \varphi~~~}&K\\
\end{array}$$
where the automorphism $\varphi $ of $\pi$ is induced by an
orientation-preserving self-diffeomorphism of $\Sigma$ fixing $\ast$
and $\tilde \varphi \in \Aut(K)$. This leads to two respective
subgroups $MC_\nabla(\Sigma;p)$ and $MC_\Box(\Sigma;p)$ of
$MC(\Sigma )$ depending upon the choice of
 $p$. We explain that these groups generalize, among others,
the Torelli groups and the handlebody  groups.

We show that most of the combinatorial structure
\cite{ABP,BKP,GM,MP} in terms of fatgraphs and flips coming from the
cell decomposition \cite{Penbook} of the moduli space of $\Sigma$
persists in this more general setting.

We then focus on the case of abelian $K$. In this case, in
generalization of \cite{MP}, an analogue of the classical Johnson
homomorphism \cite{Morita99} is expressed by an explicit cocycle in
terms of fatgraphs whose edges are  labeled by elements of $K$.
Likewise, we define a combinatorial cocycle of the Ptolemy groupoid
with coefficient in $K$. In the case $K=H$, the first homology of
the surface, we further show that it descends to $-2$ times the
Earle class \cite{Ea} \cite{Morita89}, which is a generator of the
cohomology group $H^1(MC(\Sigma);H)$.

We emphasize the abelian case  here in order to extend the
combinatorial theory already developed for the mapping class groups
and Torelli groups.  The   non-abelian case remains largely open and
should be interesting to explore. Indeed, there are many open
questions in this area, some of which are analogues of the classical
problems while others are specific  to the new technology. Several
questions and problems are raised throughout the paper. For
concreteness, we restrict ourselves to surfaces with one boundary
component, but the general case may be addressed similarly.

\section{Notation and Background}

\subsection{The Mapping Class Group} We introduce notation which
will be used throughout the paper. We fix a  compact oriented smooth
surface $\Sigma=\Sigma_{g,1}$
 of genus $g\geq 1$ with one boundary component   and fix a basepoint $*\in\partial \Sigma$.  Let
$\pi=\pi_1(\Sigma,*)$ be the fundamental group of $\Sigma$ and
$\zeta\in\pi$ be the homotopy class of the loop $\partial \Sigma$
with orientation {\it opposite} to the one induced by that of $
\Sigma$.   Let $H=H_1(\Sigma;{\mathbb Z}) \cong \ZZ^{2g}$ be the
first integral homology group of $\Sigma$.    The
 homology intersection pairing in $H $
is  skew-symmetric and non-degenerate. It is denoted $\cdot $.

Let $MC(\Sigma)$ denote the {\it mapping class group} of isotopy
classes of   diffeomorphisms of $\Sigma$ that fix $\partial \Sigma$
pointwise. We recall the following   standard facts about this
group.

\vskip .2cm

\begin{fact}\label{fact1}
Any $\varphi\in MC(\Sigma)$ induces an automorphism $\varphi_\ast
\in \Aut(\pi)$ of $\pi$ in the obvious way. The Dehn-Nielsen theorem
\cite{Zieschang} asserts that the map $\varphi\mapsto \varphi_\ast$
is an embedding of  $MC(\Sigma)$ onto the subgroup of $\Aut(\pi)$
fixing $\zeta$.  We shall   identify $\varphi\in MC(\Sigma)$ with
$\varphi_\ast\in \Aut(\pi)$.
\end{fact}

\vskip .2cm

\begin{fact}\label{fact2} Let $a:\pi\to H$ be the abelianization
map.
Any homomorphism $p $ from $\pi$ to an abelian group $K$ induces a
unique homomorphism $p_\#:H \to K$ such that  $p=p_\# \circ a :\pi
\to K$.
 In particular,
any $\varphi\in \Aut(\pi)$  induces an isomorphism $\varphi_\#:H\to
H$ such that the following diagram commutes:
$$\begin{array}{cccccc}
&& \pi&\xrightarrow{\varphi}&\pi\\\\
&&a~\downarrow&&\downarrow~a\\\\
&&H&\xrightarrow{\varphi_\#}&H \, .\\
\end{array}$$
If $\varphi\in MC(\Sigma)$ (i.e., if $\varphi(\zeta)=\zeta$), then
$\varphi_\#$ preserves the intersection pairing.  The resulting
homomorphism  from $MC(\Sigma) $ to the symplectic group $Sp({H})$
is an epimorphism, cf.\ \cite{MKS}.
\end{fact}

\subsection{The Torelli-Johnson-Morita theory}

The lower central series $\pi=\Gamma_0\supset \Gamma_1 \supset
\Gamma_2\supset \cdots $ of $\pi$ is recursively defined by taking
commutator groups $\Gamma_{k+1}=[\pi, \Gamma_k]$ for all $k\geq 0$.
The $k$-th nilpotent quotient of $\pi$ is given by
$N_k=\pi/\Gamma_k$, for $k\geq 0$. We have a short exact sequence
$$0\to{\mathcal L}_{k+1}\to N_{k+1}\to N_k\to 1,$$ where ${\mathcal L}_{k+1}=\Gamma_k/\Gamma_{k+1}$ is naturally identified with the degree $k+1$ part of the free Lie algebra
generated by $H$, cf.\  \cite{Labute, Morita99}. Note that $N_0=0$
and ${\mathcal L}_{1}=N_1=H$.

For each $k\geq 0$, denote by   $MC_k(\Sigma)$   the {\it $k$-th
Torelli subgroup of $MC(\Sigma)$} consisting of those mapping
classes that act identically on $N_k$. We obtain thus the {\it
Johnson filtration}
$$MC(\Sigma)=MC_0(\Sigma) >  MC_1(\Sigma) >  MC_2(\Sigma) > \cdots$$
of $MC(\Sigma)$. (Notation $F>G$ means in this context that $G$ is a
normal subgroup of $F$.)  The group $MC_1(\Sigma)$ is the {\it
classical Torelli group}.

For each $k\geq 1$, there is a {\it Johnson homomorphism}
\begin{equation}\label{JH}
\tau_k:MC_k(\Sigma)\to {\rm Hom}(N_{k+1},{\mathcal L}_{k+1})
\end{equation}
defined as follows.  If $\varphi\in MC_k(\Sigma)$ and $x\in
N_{k+1}$, then   $\varphi(x) x^{-1}\in N_{k+1}$ projects to  $0\in
N_k $. By   the exact sequence above, $\varphi(x) x^{-1}$ is the
image of a unique     $y_x\in {\mathcal L}_{k+1}$. The formula
$x\mapsto y_x$ defines a map $N_{k+1}\to {\mathcal L}_{k+1}$ denoted
$\tau_k(\varphi)$. The map $\tau_k(\varphi):N_{k+1}\to {\mathcal
L}_{k+1}$ is
  a homomorphism because for any $x_1, x_2\in N_{k+1}$,
$$
\varphi(x_1x_2)(x_1x_2)^{-1}=\varphi(x_1)~(\varphi(x_2)x_2^{-1})~x_1^{-1} = (\varphi(x_1)x_1^{-1})~(\varphi(x_2)x_2^{-1}).
$$
The second equality uses the fact that $\varphi(x_2)x_2^{-1}$ is
central in $N_{k+1}$ as it lies in the kernel of the projection
$N_{k+1}\to N_k$.

\begin{claim}\label{cl2} The map \eqref{JH} is a homomorphism.
\end{claim}

\begin{proof}
For any $\varphi, \psi \in MC_k(\Sigma)$, we have
$$\aligned
(\varphi\circ\psi(x))x^{-1}&=\varphi(\psi(x))x^{-1}\\
&=\varphi(\psi(x)x^{-1})~\varphi(x)x^{-1}\\
&=\psi(x)x^{-1}~\varphi(x)x^{-1}=\varphi(x)x^{-1}~\psi(x)x^{-1},\\
\endaligned$$
where the last equality follows from the commutativity of ${\mathcal
L}_{k+1}$   and the next-to-last equality holds because ${\mathcal
L}_{k+1}$ is pointwise invariant under the automorphism of $N_{k+1}$
induced by $\varphi$. This is so because $\varphi \in
MC_k(\Sigma)\subset MC_1(\Sigma)$ acts as the identity in
$H=\Gamma_0/\Gamma_1$ and therefore acts as the identity on all the
quotients $\Gamma_k/\Gamma_{k+1}$.
\end{proof}

By    construction,  the kernel of the Johnson homomorphism $\tau_k$
is $MC_{k+1}(\Sigma)$. To study the image of $\tau_k$, we identify
$$
 {\rm Hom}(N_{k+1},{\mathcal L}_{k+1})
\cong {\rm Hom}(H,{\mathcal L}_{k+1})
\cong  H^*\otimes{\mathcal L}_{k+1}
\cong H\otimes{\mathcal L}_{k+1}
$$
  where we  use  the fact that ${\mathcal
L}_{k+1}$ is abelian in the first isomorphism and the Poincar\'e
duality in the last isomorphism.

\begin{claim}\label{cl3} \cite{Morita93}
The image of $\tau_k$ in $H\otimes{\mathcal L}_{k+1}$ lies in the
kernel of the additive bracket map $ b_k:H\otimes{\mathcal L}_{k+1}
\to {\mathcal L}_{k+2}$ carrying $x\otimes u$ to $[x,u]$ for all
$x\in H$ and $u\in {\mathcal L}_{k+1}$.
\end{claim}

\subsection{Fatgraphs and spines}

Let $G$ be a finite connected 1-dimensional CW-complex   with the
set of vertices $V(G)$, the set of edges $E(G)$ and the set of
oriented edges $\tilde E(G)$. Reversing the orientation of an edge,
we obtain a canonical involution $e\mapsto \bar e$ on $\tilde E(G)$.
An edge of the first barycentric subdivision of $G$ is a {\it
half-edge}.   Given a half-edge containing a vertex $v$ of $G$, we
assign the ambient edge in $E(G)$ the orientation towards $v$. In
this way, we identify the set of half-edges of $G$ with $\tilde
E(G)$.

The number of half-edges incident on a vertex of $G$ is the {\it
valence} of this vertex. We shall always assume that each vertex of
$G$ has valence at least three except for a single univalent vertex,
whose incident edge is called the {\it tail}. The tail has a
preferred orientation $t \in \tilde E(G)$ pointing away from its
univalent endpoint. We call $G$ a  {\it fatgraph (with tail)} if it
is endowed with the additional structure given by a family of cyclic
orderings, one such cyclic ordering on each collection of half-edges
incident to a common vertex.

A fatgraph $G$ determines a  compact connected oriented   surface
$\Sigma(G)$ as follows. For each vertex $v\in V(G)$ of valence $k$,
consider an oriented polygon $P_v$ of $2k$ sides, where every other
consecutive side corresponds to a half-edge incident on $v$ and
where the given cyclic ordering about $v$ corresponds to traversing
the boundary of $P_v$ with the polygon on the left. The surface
$\Sigma(G)$ is the quotient of the disjoint union $\sqcup_{v\in
V(G)} P_v$,  where the sides of polygons are identified by an
orientation-reversing homeomorphism if the corresponding half-edges
lie in a common edge of $G$. By the \emph{genus}  and the
\emph{boundary number of} $G$ we shall mean respectively the genus
of $\Sigma(G)$ and the number of components of $\partial \Sigma(G)$.
Note that if the boundary number of $G$ is equal to $1$, then there
is a canonical strict order on $\tilde E(G)$ determined by order of
appearance when traversing $\partial \Sigma(G)$ starting from the
tail $t$.

A {\it fatgraph spine} of $\Sigma=\Sigma_{g,1}$ is a  pair
consisting of a fatgraph (with tail) $G$ and the isotopy class
relative to the boundary of an embedding $i:G\to \Sigma$ such that
$i(G)$
  is a strong deformation retract of $\Sigma$,    the
cyclic orderings about vertices of $G$ agree with the
counter-clockwise sense in~$\Sigma$, and
 $i^{-1}(\partial \Sigma)=i^{-1}(\partial \Sigma-\{ *\})$ is the univalent vertex of~$G$.
 Any such   $i$ extends to an orientation-preserving
 homeomorphism $\Sigma(G)\approx \Sigma$.
 Thus, the fatgraph $G$
has genus $g$ and boundary number 1.

\subsection{  Mapping Class Groupoids} The mapping class group $MC(\Sigma)$ acts without
isotropy on the contractible {\it Teichm\"uller space} $T(\Sigma)$
of isotopy classes of hyperbolic structures with geodesic boundary
on $\Sigma$. The  quotient $ T(\Sigma)/MC(\Sigma)$ is the {\it
Riemann moduli space} of $\Sigma$, cf.\ \cite{Penbook}.

\begin{theorem} \cite{Penbook} \label{thmone}
For  $\Sigma=\Sigma_{g,1}$ with any $g\geq 1$, there is a $MC(\Sigma
)$-invariant ideal simplicial decomposition of the Teichm\"uller
space $T(\Sigma )$, where cells are indexed by fatgraph spines of
$\Sigma$  and the face relation is given by contracting non-tail
edges with distinct endpoints.
\end{theorem}

\begin{figure}[!h]
\begin{center}
\epsffile{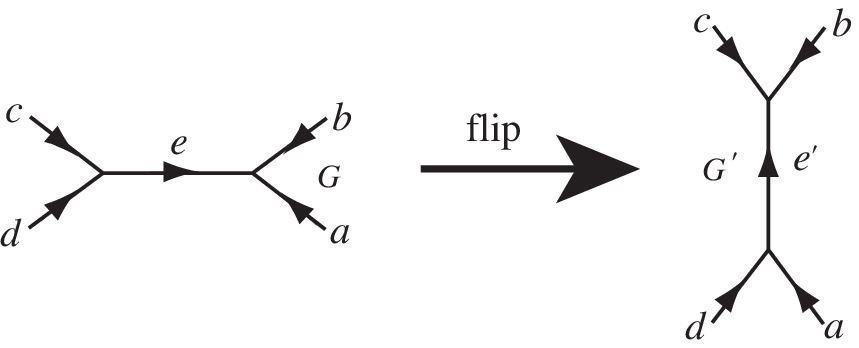}
\caption{A flip}
\label{figone}
\end{center}
\end{figure}

The dual of this decomposition of $T(\Sigma)$ is   the {\it dual
fatgraph complex} ~$\hat{\mathcal G}=\hat{\mathcal G}(\Sigma)$.
The vertices of $\hat{\mathcal G}$ correspond to trivalent fatgraph
spines, and the oriented 1-simplices correspond to {\it flips},
namely, the collapse and distinct expansion of a non-tail edge with
distinct endpoints in a trivalent fatgraph spine, cf.\ Figure
\ref{figone}. Non-degenerate 2-cells in $\hat{\mathcal G}$
correspond to fatgraph spines   whose vertices are trivalent except
for either one 5-valent vertex or two 4-valent vertices.

Suppose that $G'$ arises from $G$ by a flip along an edge $e\in
E(G)$ and let $e'\in E(G')$ denote the corresponding edge as in
Figure \ref{figone}.  This correspondence extends to an
identification of $E(G)$ with $E(G')$ in the obvious way.  This
lifts to an identification of $\tilde E(G)$ with $\tilde E(G')$ by
demanding that the oriented edges   $e,e'$   in this order agree
with the orientation of $\Sigma$ and the  orientation of the other
edges is unchanged as in Figure \ref{figone}.

Recall that the {\it fundamental path groupoid} of a CW-complex $X$ is
a category whose objects are the vertices of $X$ and whose morphisms
are the homotopy classes of edge-paths in $X$. The {\it Ptolemy
groupoid} $Pt(\Sigma)$ of $\Sigma$ is the fundamental path groupoid
of $\hat{\mathcal G} (\Sigma)$.

\begin{corollary} \cite{Penbook} \label{corone}
For any $g\geq 1$, the objects of the Ptolemy groupoid  $Pt(\Sigma)$
of $\Sigma=\Sigma_{g,1}$ are the trivalent fatgraph spines of
$\Sigma$, and the morphisms are   finite sequences of flips subject
to the following relations:

\vskip .2cm

\leftskip=.4cm\rightskip=.4cm

\noindent {\rm [Involutivity]}~ the flip along an edge followed by the flip along its corresponding edge yields the identity;

\vskip .2cm

\noindent {\rm [Commutativity]}~  the flips along disjoint edges
commute;

\vskip .2cm

\noindent
{\rm [Pentagon Relation]}~ if two edges share a single endpoint, then the unique
 sequence of (five) flips along them and their corresponding edges that contains no involutive pair of flips yields the identity.

\end{corollary}

The action of  $MC(\Sigma)$ on  $T(\Sigma)$ induces  a fixed point
free action of $MC(\Sigma)$ on the CW-complex $\hat{\mathcal G} $.
 The {\it mapping class groupoid} $M\Gamma(\Sigma)$ of $\Sigma$
is the fundamental path groupoid of the quotient CW-decomposition $
\hat{\mathcal G}(\Sigma)/MC(\Sigma)$ of the Riemann moduli space of
$\Sigma$. The objects of $M\Gamma(\Sigma)$  are the
$MC(\Sigma)$-orbits of trivalent fatgraph spines of $\Sigma$ and the
morphisms are generated by $MC(\Sigma)$-orbits of flips subject to
$MC(\Sigma)$-orbits of the relations above.  The mapping class group
$MC(\Sigma)$ is isomorphic to the group of automorphisms of any
object of $M\Gamma(\Sigma)$.

\section{Markings}

Throughout this section we fix  a group $K$.

\subsection{Markings on $\pi$}\label{auto} A {\it $K$-marking} on
$\pi=\pi_1(\Sigma,*)$ is an epimorphism $\pi\xrightarrow{\it p}K$. A
{\it $\nabla$-automorphism} of a marking $p$ is an automorphism
$\varphi$ of $\pi$ such that the following diagram commutes:
$$\begin{array}{cccccc}
&\pi&&\xrightarrow{~~~\varphi~~~}&&\pi\\\\
&&p~\searrow&&\swarrow~p\\\\
&&&K\, .\\
\end{array}$$
A {\it $\Box$-automorphism} of $p$ is  an automorphism $\varphi$ of $\pi$ such that $\varphi (\Ker p)=\Ker p$.
Such a $\varphi$ induces   an automorphism $\tilde \varphi$ of $K$ making the
following diagram commute:
$$\begin{array}{cccccc}
&& \pi&\xrightarrow{~~~~\varphi~~~~}&\pi\\\\
&&p~\downarrow&&\downarrow~p\\\\
&&K&\xrightarrow{~~~\tilde \varphi~~~}&K\,.\\
\end{array}$$
 The $\nabla$--automorphisms (resp.\@    $\Box$-automorphisms) of $p$
   form a group under composition,
$\Aut_\nabla(p)$ (resp.,\@ $\Aut_\Box(p)$). We have an obvious
exact sequence
$$1\to \Aut_\nabla(p)\to \Aut_\Box(p)\to \Aut(K).$$
In general, the projection $\Aut_\Box(p)\to \Aut(K)$ is not
surjective.

A $\nabla$- or a  $\Box$-automorphism of $p$ is {\it topological} if
it  is realizable by a
self-diffeomorphism  of $\Sigma$ that fixes $\partial \Sigma$
pointwise. We define the {\it $p$-relative $\nabla$- and
$\Box$-mapping class groups} of $\Sigma$ to be the respective
subgroups of topological automorphisms
$$\aligned
MC_\nabla(\Sigma;p)<\Aut_\nabla(p) \qquad {\rm and} \qquad  MC_\Box(\Sigma;p)<\Aut_\Box(p)\, .
\endaligned$$
  Clearly, $MC_\nabla(\Sigma;p)\leq MC(\Sigma)$ and
$MC_\Box(\Sigma;p)\leq MC(\Sigma) \rtimes \Aut(K)$.

\begin{example}
If $K=(1)$, then $MC_\Box(\Sigma;p)=MC_\nabla(\Sigma;p)$
$=MC(\Sigma)$.
\end{example}

\begin{example}\label{a_k}
For all $k\geq 1$ and the   projection $\pi\xrightarrow{a_k}N_k$, we
recover the $k$-th Torelli group
 $MC_\nabla(\Sigma;a_k)=MC_k(\Sigma)$ while $MC_\Box(\Sigma;a_k)=MC(\Sigma)$.
 In fact,
$MC_\nabla(\Sigma;p)=MC_k(\Sigma)$ and
$MC_\Box(\Sigma;p)=MC(\Sigma)$ for any $N_k$-marking
$\pi\xrightarrow{p} N_k$ because   such a marking is the composition
of $a_k$ with an automorphism of $N_k$. This  follows from   the
Hopfian property of~$N_k$, see \cite{MKS}, Section 5.3, Theorem 5.5.
\end{example}

\begin{example}
Let $F=\pi_1(A, \ast)$ where $A$ is a genus $g$ handlebody bounded
by  $\Sigma$ capped off by a 2-disk. If $p:\pi\to F$ is the
inclusion homomorphism, then   $MC_\Box(\Sigma;p)$ is the
corresponding handlebody group relative to a disk, cf. \cite{MCM},
\cite{Suzuki}.
\end{example}

The following proposition is a direct consequence of   Fact
\ref{fact2}.

\begin{proposition}
For any abelian group $K$ and any $K$-marking $p$ on $\pi $,
 we have the short exact sequence
$$1\to MC_\nabla(\Sigma;p)\to MC_\Box(\Sigma;p)\to I_K^H\to 1,$$
where $I_K^H$ is the group of all  $\psi\in \Aut(K)$ such that there
is a symplectomorphism $\Psi\in Sp(H)$ making the diagram
$$\begin{array}{cccccc}
&& H&\xrightarrow{~~~~\Psi~~~~}&H\\\\
&&{p_\#}~\downarrow&&\downarrow~{p_\#}\\\\
&&K&\xrightarrow{~~~\psi~~~}&K\\
\end{array}$$
commute (here  ${p_\#}$ is induced by $p$ as in Fact \ref{fact2}).
\end{proposition}

\begin{question}
Is there a better description of the group
$I_K^H$?
\end{question}

\begin{question}
When are $MC_\nabla(\Sigma;p)$ or $MC_\Box(\Sigma;p)$ finitely generated
and/or finitely presented?
\end{question}

\subsection{Markings on fatgraphs}  Following \cite{MP,BKP}, we define a {\it $K$-marking} on a fatgraph $G$ to be     a map
$\mu:\tilde E(G)\to K$ such that \vskip .2cm \leftskip
.8cm\rightskip .8cm

\noindent [Inversion]  for any $e\in \tilde E(G)$, we have $\mu(\bar
e)= \mu (e)^{-1}$;

\vskip .2cm

\noindent [Coherence] if $e_1,\ldots ,e_k$ are the half-edges
incident on a common $k$-valent vertex of $G$ in a linear order
consistent with the given cyclic order, then  $\mu(e_1)\cdots
\mu(e_k)=1\in K$;

\vskip .2cm

\noindent [Surjectivity] the set $\{ \mu(e):e\in\tilde E(G)\}$
generates $K$.

\leftskip=0ex\rightskip=0ex

\vskip .2cm

\begin{example}\label{exampleMarking}
A fatgraph spine $i:G\to \Sigma$ of $\Sigma $ determines a
$\pi$-marking on  $G$ as follows.  Pick an oriented edge $e\in
\tilde E(G)$ of $G$, and take a small arc $\alpha_e$ transverse to
$e$ in $\Sigma$. We orient $\alpha_e$ so that  the positive tangent
vectors of  $\alpha_e$ and $e$   (in this order) form a positively
oriented basis in the tangent space of $\Sigma$. Since the
complement of $i(G)$ is a contractible set
 containing $*\in\partial \Sigma$, there are unique homotopy classes of paths from the endpoints of $\alpha_e$
  to $*$ in $\Sigma-i(G)$. These paths and $\alpha_e$ combine into   a loop in $\Sigma$ based at $*$.
 The assignment of the homotopy class of this loop to   $e $ determines a   $\pi$-marking on $G$.
\end{example}

  Each fatgraph spine $i:G\to \Sigma $ of $\Sigma$ determines a
  bijection between the $K$-markings on $\pi$ and the $K$-markings
  on $G$. This bijection carries a $K$-marking
  $\pi\xrightarrow{p} K$  on $\pi$ to the
$K$-marking   $p \mu : \tilde E(G)\to K$ on $G$, where $\mu:\tilde
E(G)\to \pi$ is the $\pi$-marking    derived from $i $ in
Example~\ref{exampleMarking}.  A $K$-marking on a fatgraph $G$ that
arises in this way from a fatgraph spine $i:G\to \Sigma $ and a
$K$-marking
  $\pi\xrightarrow{p} K$ is
said to be {\it topological  with respect to $p$}.

Flips  along edges act naturally on marked trivalent fatgraphs.
Consider a $K$-marking $\mu:\tilde E(G)\to K$ on a trivalent
fatgraph $G$ and consider a flip along the underlying unoriented
edge of an oriented non-tail edge $e\in \tilde E(G) $. Let $G'$ be
the resulting fatgraph   with   edge $e'$ corresponding to $e$.  The
marking $\mu$ induces a unique $K$-marking $\mu'$ on $G'$ equal to
$\mu $ on $\tilde E(G)-\{ e,\bar e\}= \tilde E(G')-\{ e',\bar e'\}$.
The coherence condition allows us  to compute $\mu'(e')$. In the
notation of Figure \ref{figone},
$$\mu(e')=\mu(d)\, \mu(a) = \mu(\bar c)\, \mu(\bar b) \quad {\rm {and}} \quad \mu'(\bar e')=\mu(\bar a)\,\mu (\bar
d)= \mu(b) \,\mu(c).$$ Observe that if the $K$-marking $\mu $ is
topological with respect to $\pi\xrightarrow{p} K$, then so is the
$K$-marking $\mu'$.

\begin{proposition} \cite{BKP}    Let $G$ be a fatgraph with genus $g$ and  boundary number 1.
A $\pi$-marking $\mu$ on $G$ is topological with
respect to the identity map $\id:\pi\to \pi$ if and only if $\mu (
 t)=\zeta^{-1}$ where $t\in \tilde E(G)$ is the tail  of~$G$.
  An $H$-marking $\mu$ on $G$ is topological with
respect to the abelianization map $\pi\to H$ if and only if it
respects homology intersection numbers in the sense that
$$\mu(a)\cdot \mu (b)=\begin{cases}
+1,&{  if}~a<b<\bar a<\bar b~{  up~to~cyclic~permutation};\\
-1,&{  if}~a<\bar b<\bar a< b~{  up~to~cyclic~permutation};\\
\hskip 1.7ex 0,&{ else}
\end{cases}$$
where $<$ denotes the strict order  on $\tilde E(G)$ given by
occurrence along   $\partial \Sigma(G)$ starting from the tail.
\end{proposition}

\begin{proof}
The first claim follows  from Fact \ref{fact1} and the Hopfian
property of   $\pi$. The second claim follows   from the last
assertion of Fact \ref{fact2}.
\end{proof}

\begin{question}
For $k\geq 2$, when is an $N_k$-marking on a fatgraph  topological?
\end{question}

\section{Marked fatgraph complexes and cocycles}

Fix a group $K$  and a $K$-marking $p:\pi=\pi_1(\Sigma, \ast)\to K$.
Insofar as   $MC(\Sigma)$
 acts on the Teichm\"uller space $T(\Sigma)$  without fixed points, so   does its subgroup $MC_\nabla(\Sigma;p)$.
We define the {\it $p$-relative moduli space}
$$
M_\nabla(\Sigma;p)=T(\Sigma)\slash MC_\nabla(\Sigma;p).\\
$$
This is an Eilenberg-MacLane space of type $K(MC_\nabla(\Sigma;p),
1)$. The ideal simplicial decomposition of $T(\Sigma)$ is likewise
invariant under this (or any) subgroup of $MC(\Sigma)$, and hence
descends to the $p$-relative moduli space. The dual of this
decomposition is the
 CW-complex
$$
\hat M_\nabla(\Sigma;p)=\hat{\mathcal G}(\Sigma)\slash MC_\nabla(\Sigma;p).\\
$$
Its fundamental path groupoid is denoted by
$M\Gamma_\nabla(\Sigma;p)$ and called the  {\it $p$-relative mapping
class groupoid}. It contains $MC_\nabla(\Sigma;p)$ as the stabilizer
of any object.

\begin{example} For the projection $a_k:\pi\to N_k$, the space $M_\nabla(\Sigma;a_k)=T(\Sigma)\slash MC_k(\Sigma)$ is
the $k$-th Torelli space of $\Sigma$. The groupoid
$M\Gamma_\nabla(\Sigma;a_k)$  contains
 $MC_k(\Sigma)$ as a subgroup of infinite index, namely,   the stabilizer of any
 object.
\end{example}

\begin{corollary} Let $\pi\xrightarrow{p} K$ be a $K$-marking on $\pi$.
 An object of
$M\Gamma_\nabla(\Sigma;p)$ is given by a trivalent fatgraph   with
tail and  a  $K$-marking on this fatgraph topological with respect
to $p$. Morphisms are given by finite sequences of flips acting on
$K$-marked fatgraphs with tail subject to the Involutivity,
Commutativity, and Pentagon Relations.
\end{corollary}

\begin{proof}
This follows from Corollary \ref{corone} and general position since
finite sequences of flips relate any two topological $\pi$-markings
on a fixed fatgraph spine with tail.
\end{proof}

Thus, an element of $MC_\nabla(\Sigma;p)$ is determined by a
sequence of flips starting from a fatgraph $G$ with tail and a
topological $K$-marking $\mu:\tilde E(G)\to K$; the sequence ends
with a combinatorially isomorphic fatgraph $G'$ with tail and a
(necessarily topological) $K$-marking $\mu':\tilde E(G')\to K$
induced from $\mu$ by the   flips   so that the bijection $\tilde
E(G) \approx \tilde E(G')$ induced by the combinatorial isomorphism
$G\approx G'$ carries $\mu$ to $\mu'$.

  Suppose now that the group $K$ is abelian.  Adopting
the notation of Figure \ref{figone} for the edges involved in a
flip, we define combinatorial (untwisted) cochains $$m^p\in
C^1(\hat{\mathcal{G}}(\Sigma); K), \quad j^p \in
C^1(\hat{\mathcal{G}}(\Sigma); \Lambda^3 K), \quad s^p \in
C^1(\hat{\mathcal{G}}(\Sigma); K)$$ by
$$
 m^p(W) =p(a)+p(c) \in K,$$
 $$
j^p(W)  = p(a)\wedge p(b) \wedge p(c) \in \Lambda^3 K,$$
$$s^p(W) =$$
$$\biggl (\bigl [p(a)\wedge p(c)\bigr ]\otimes\bigl [p(b)\wedge
p(d)\bigr ]\biggr )+ \biggl (\bigl [p(b)\wedge p(d)\bigr
]\otimes\bigl [p(a)\wedge p(c)\bigr ]\biggr )   \in S^2 \Lambda^2 K.
$$
 Note that the obvious action of the group $MC_{\Box}(\Sigma;p)$
  on $K$  induces an action of $MC_{\Box}(\Sigma;p)$
 on   $\Lambda^3 K$ and $S^2 \Lambda^2 K$. We say that a 1-cochain $c$ on $\hat{\mathcal{G}}(\Sigma)$ with values in one of these groups is
 $MC_{\Box}(\Sigma;p)$-equivariant if $c(\varphi(\sigma))=\varphi(c(\sigma))$ for any oriented 1-cell
$\sigma$ of $\hat{\mathcal{G}}(\Sigma)$ and $\varphi \in
MC_{\Box}(\Sigma;p)$.

\begin{theorem}\label{thm:cocycle}
The  cochains $m^p$, $j^p$, and $s^p$  are
$MC_{\Box}(\Sigma;p)$-equivariant cocycles on
$\hat{\mathcal{G}}(\Sigma)$.  In particular, they descend  to
1-cocycles on the cell complex $\hat M_\nabla(\Sigma;p)$ with the
corresponding coefficients.
\end{theorem}

\begin{proof}
One first checks that the cochains are well-defined, i.e., are
independent of the orientation on the edge along which the flip $W$
is performed.  The Commutativity Relation leads to vanishing
expressions since the groups $K$, $\Lambda ^2 K$, and $S^2\Lambda ^2
K$ are abelian.  Vanishing of the expressions for the Involutivity
is easily proved by direct computations.

\begin{figure}[!h]
\begin{center}
\epsffile{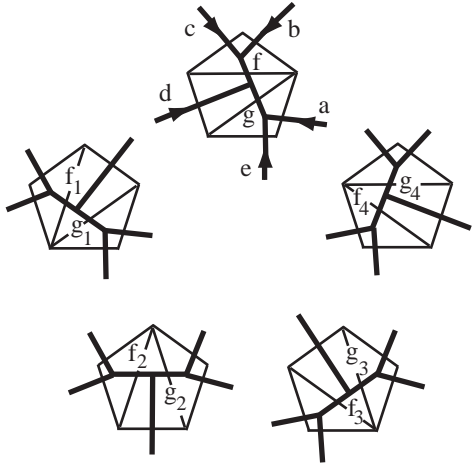}
\caption{Notation for pentagon relation}
\label{fig:pent}
\end{center}
\end{figure}

In the notation of  Figure \ref{fig:pent}   reproduced from
\cite{MP},  we shall check the Pentagon Relation
$W_fW_{g_1}W_{f_2}W_{g_3}W_{f_4}=1$ in each case. For simplicity, we
shall write $a$ instead of $p(a)$, $b$ instead of $p(b)$, etc.

Pentagon Relation for $m=m^p$. We have $m(W_f)=b+d$, $m(W_{g_1})=b+e$,
$m(W_{f_2})=c+e$, $m(W_{g_3})=a+c$, $m(W_{f_4})=a+d$. Therefore,
$$m(W_f)+m(W_{g_1})+m(W_{f_2})+m(W_{g_3})+m(W_{f_4})$$
$$= 2(a+b+c+d+e)=0,$$ as desired, where the last equality follows from
Coherence.

Pentagon Relation for $j=j^p$: see Theorem 3.2  of  \cite{MP}.

Pentagon Relation for $s=s^p$.  By definition,
\begin{align*}
s(W_f) &= [b\wedge d] \otimes [c\wedge g]+[c\wedge g]\otimes [b\wedge d], \\
s(W_{g_1}) &= [e\wedge b]\otimes [a\wedge f_1]+[a\wedge f_1]\otimes [e\wedge b], \\
s(W_{f_2}) &= [e\wedge c]\otimes [g_2\wedge d]+[g_2\wedge d]\otimes [e\wedge c], \\
s(W_{g_3}) &= [a\wedge c]\otimes [b\wedge f_3]+[b\wedge f_3]\otimes [a\wedge c], \\
s(W_{f_4}) &= [a\wedge d]\otimes [g_4\wedge e]+[g_4\wedge e]\otimes [a\wedge d].
\end{align*}
By Coherence,   $g=a+e$, $f_1=c+d$, $g_2=a+b$, $f_3=d+e$, $g_4=b+c$.
Hence
\begin{align*}
& s(W_f)+s(W_{g_1})+s(W_{f_2})+s(W_{g_3})+s(W_{f_4}) \\
= &   [b\wedge d]\otimes [c\wedge (a+e)]+[c\wedge (a+e)]\otimes [b\wedge d] \\
+& [e\wedge b]\otimes [a\wedge (c+d)]+[a\wedge (c+d)]\otimes [e\wedge b] \\
+& [e\wedge c]\otimes [(a+b)\wedge d]+[(a+b)\wedge d]\otimes [e\wedge c] \\
+& [a\wedge c]\otimes [b\wedge (d+e)]+[b\wedge (d+e)]\otimes [a\wedge c] \\
+& [a\wedge d]\otimes [(b+c)\wedge e]+[(b+c)\wedge e]\otimes [a\wedge d].
\end{align*}
Direct cancellation using the defining properties of $\wedge$ and
$\otimes$ shows that the latter sum vanishes.    The
$MC_{\Box}(\Sigma;p)$-equivariance of the cocycles is clear.
\end{proof}

As a consequence,  the  cocycles $m^p$, $j^p$, and $s^p$
  give rise to twisted first cohomology classes of
$MC_{\Box}(\Sigma;p)$ and untwisted cohomology classes of
$MC_{\nabla}(\Sigma;p)$.   Indeed, fix    a vertex $v_0$ of the
complex $\hat{\mathcal{G}}(\Sigma)$, i.e., a fatgraph spine of
$\Sigma$. For $\varphi \in MC_{\Box}(\Sigma;p)$, choose a chain
$\sigma_1,\ldots,\sigma_n$ of oriented 1-cells connecting $v_0$ to
$\varphi(v_0)$ and set
$$\tilde{m}^p (\varphi)=\sum_{i=1}^n m^p(\sigma_i).$$
 The fact that $m^p$ is a cocycle ensures that $\tilde{m}^p
(\varphi)$ is well-defined. The $MC_{\Box}(\Sigma;p)$-equivariance
of $m^p$ implies the following cocycle condition for $\tilde{m}^p$:
for any $\varphi, \psi \in MC_{\Box}(\Sigma;p)$, we have
$$\tilde{m}^p(\varphi \psi)=\tilde{m}^p(\varphi)+\varphi \cdot \tilde{m}^p(\psi) $$
(this construction of a twisted 1-cocycle appeared in \cite{MP}, \S
3). The cocycle  $\tilde{m}^p$ depends on the choice of   $v_0$  but
its cohomology classes  does not depend on this choice. We obtain
thus a twisted cohomology class  $[\tilde{m}^p]\in
H^1(MC_{\Box}(\Sigma;p); K)$. The cocycles $j^p$  and $s^p$
similarly give  rise to   twisted cohomology classes
  $$[\tilde{s}^p]\in
H^1(MC_{\Box}(\Sigma;p); \Lambda^3 K)\quad {\rm {and}} \quad [\tilde{j}^p]\in
H^1(MC_{\Box}(\Sigma;p); S^2 \Lambda^2K).$$ These three cohomology
classes restrict to respective homomorphisms
$$  MC_\nabla(\Sigma;p)\to K, \quad MC_\nabla(\Sigma;p)\to
\Lambda ^3 K, \quad   MC_\nabla(\Sigma;p)\to S^2\Lambda^2 K.$$

The cocycle $m^p$ is new.  When $p=a:\pi\to H$ is the abelianization
map, the cocycle $j^p=j^a$ was   discovered in \cite{MP}. It is
proven in \cite{MP,BKP} that
 $j^a$ represents six times the
first Johnson homomorphism~$\tau _1$. The cocycle $s^a$ was
discovered in \cite{Penbook}. One can   show that the cohomology
class of $s^a$ vanishes. In \S \ref{sec:m}, we show that
$[\tilde{m}^a]\in H^1(MC(\Sigma);H)\approx{\mathbb Z}$ is $-2$ times
a generator called  the Earle class.

\begin{question}
Are there abelian markings $p$ so that $s^p $ is   non-trivial?
\end{question}

\begin{question}
What the are functorial properties in $p$ of   $  m^p,s^p, j^p$?
\end{question}

\begin{question}
For which $p$ is $j^p$ non-trivial?  Mimicking a  computation in
\cite{MP}, one can compute the value of $j^p$ on a BP torus map  and
find this value to be non-vanishing if a local condition holds
thereby addressing this question provided the BP map lies in
$MC_\nabla(\Sigma;p)$.
\end{question}

\section{Relative Torelli-Johnson-Morita Theory}

Fix a group $K$  and a $K$-marking $\pi\xrightarrow{p}K$ on $\pi$.
The group  $K$ has its own lower central series
$\Gamma_{k+1}(K)=[K,\Gamma_k(K)]$, where $\Gamma_0(K)=K$, and the
nilpotent quotients $N_k(K)=K/\Gamma_k(K)$. For all $k\geq 0$, we
have a short exact sequence
$$0\to{\mathcal L}_{k+1}(K) \to N_{k+1}(K)\to N_k(K)\to 1,$$
where ${\mathcal L}_{k+1}(K)=\Gamma_k(K)/\Gamma_{k+1}(K)$.

Composing $\pi\xrightarrow{p}K$ with the projection $K\to N_k(K)$,
we obtain  a     $N_k(K)$-marking $\pi\xrightarrow{p_k}N_k(K)$  for
all $k\geq 0$.  The group  $MC_\nabla(\Sigma; p_k)$ plays the
$p$-relative role of the $k$-th Torelli group.

\begin{example} For  $p=id:\pi\to\pi$, we have
$MC_\nabla(\Sigma;p_k)=MC_\nabla(\Sigma;a_k)$ where $a_k:\pi\to N_k$
is the canonical projection.
\end{example}

Fix $k\geq 1$. If  $\varphi\in MC_\nabla(\Sigma;p_k)$ and $x\in\pi$,
then $p_{k+1}(\varphi(x) x^{-1})\in N_{k+1}(K)$   projects to $1\in
N_k(K)$ and determines an element $y_x$ of ${\mathcal L}_{k+1}(K)$.
In generalization of the classical case, the formula
$\varphi\mapsto(x\mapsto y_x)$ defines the {\it $p$-relative $k$-th
Johnson map}
$$\tau_k^p:MC_\nabla(\Sigma; p_k)\to~{\rm Hom}(H,{\mathcal L}_{k+1}(K))
\approx H\otimes {\mathcal L}_{k+1}(K).$$
  In general,   $\tau^p_k$
does  not seem to be a homomorphism. We have the following partial
result in this direction.

\begin{proposition}
For all $k\geq 1$, the restriction of $\tau_k^p$ to the  group
$$M_k =MC_\nabla(\Sigma;
p_k) \cap MC_\Box(\Sigma; p_{k+1})\subset MC_\nabla(\Sigma;
p_k) $$ is a homomorphism.
\end{proposition}

\begin{proof} For $\varphi,\psi \in M_k$ and $x\in \pi$ we compute
\begin{align*}
p_{k+1}((\varphi \circ \psi)(x)x^{-1})
 &= p_{k+1}(\varphi(\psi(x)x^{-1})\varphi(x)x^{-1}) \\
 &=\tilde \varphi(p_{k+1}(\psi(x)x^{-1}))\, p_{k+1}(\varphi(x)x^{-1}) \\
 &=p_{k+1}(\psi(x)x^{-1}) \, p_{k+1}(\varphi(x)x^{-1}) \\
 &=(\tau_k^p(\varphi)+\tau_k^p(\psi))(x).
\end{align*}
In this computation: $\tilde \varphi$ is the automorphism of
$N_{k+1}(K)$ induced by $\varphi$; the second equality uses the
inclusion $\varphi \in MC_\Box(\Sigma;p_{k+1})$; and the third
equality follows from the fact that $\varphi  \in M_k$ acts
trivially on $H_1(K)  \cong H_1(N_{k+1}(K)) $ and therefore acts
trivially on $ \mathcal{L}_{k+1}(N_{k+1}(K))   \cong
\mathcal{L}_{k+1}(K )$ $=\Gamma_k(K)/\Gamma_{k+1}(K)$.
\end{proof}

We expect a version of Claim \ref{cl3} to   hold   for the map
$$
\bar \tau_k^p: MC_\nabla(\Sigma; p_k)\to~N_1(K)\otimes{\mathcal
L}_{k+1}(K)$$ obtained by  composing $\tau^p_k$ with the tensor
product of   $p_\#:H\to N_1(K)$ and the identity on
${\mathcal L}_{k+1}(K)$. We however do not pursue this line of thought here.

\begin{question}  What is the connection between the homomorphism
 $ MC_\nabla(\Sigma; p_1)\to~
\Lambda^3 N_1(K)$  produced by Theorem \ref{thm:cocycle} and the
morphism $$\bar \tau_1^p:MC_\nabla(\Sigma; p_1)\to~ N_1(K)\otimes
{\mathcal L}_{2}(K)\approx N_1(K)\otimes\Lambda^2 N_1(K)?$$
\end{question}

\begin{remark}
The paper \cite{BKP} shows that for an abelian $K$, any $K$-marking $\mu:\tilde E(G) \to
K$ on a trivalent fatgraph $G$ with tail   naturally extends  to a marking $\hat\mu:\tilde E(G) \to
\hat T(K)$ by the completed tensor algebra $\hat T(K)$ of $K$ so
that $\hat\mu(e)\equiv 1+\mu(e)$ modulo terms of degree at least
two for all $e\in \tilde E(G)$. This was done in order to obtain explicit formulae for
lifts of the higher Johnson homomorphisms to the classical Torelli
groupoid using Kawazumi's   description \cite{Kaw1} of these homomorphisms in
terms of generalized Magnus expansions.  It would
be interesting to   generalize these results to the
$p$-relative case.
\end{remark}

\begin{remark} As pointed out by J{\o}rgen Ellegaard Andersen, it is intriguing and appealing
to think about $MC(\Sigma_1)$-markings on a fatgraph spine of  $ \Sigma_2$ as describing
holonomies of $\Sigma_1$-bundles over $\Sigma_2$ in order to probe 4-manifolds.
\end{remark}

\section{Evaluation of the cocycle $m^a$}
\label{sec:m}

Let $\pi \xrightarrow{a} H$ be the abelianization map. The cocycle
$m=m^a$ given in Theorem \ref{thm:cocycle} determines a twisted
cohomology class $[\tilde{m}^a] \in H^1(MC(\Sigma); H)$.   Morita
\cite{Morita89} proved that $H^1(MC(\Sigma); H) $ is an infinite
cyclic group. In this section, we show that $[\tilde{m}^a]$ is twice
a generator.

 We first review the construction by Morita \cite{Morita89}.
Let $F(\alpha,\beta)$ be the free group of rank two, generated by $\alpha$, $\beta$. Any element $x\in
F(\alpha,\beta)$ can be uniquely written as
$x=\alpha^{\varepsilon_1}\beta^{\delta_1}\cdots
\alpha^{\varepsilon_n}\beta^{\delta_n}$, where
$\varepsilon_i,\delta_i \in \{ 0,\pm 1\}$. Set
$$d(x):=\sum_{k=1}^n \varepsilon_k \sum_{\ell=k}^n \delta_{\ell}
-\sum_{k=1}^n \delta_k \sum_{\ell=k+1}^n \varepsilon_{\ell}.$$ Next,
let $\{ \alpha_i, \beta_i \}_{i=1}^g\subset \pi$ be the standard
geometrically symplectic basis of Figure \ref{fig:basis} and for
$1\le i\le g$, define a group homomorphism $p_i: \pi \to
F(\alpha,\beta)$ by $p_i(\alpha_i)=\alpha$, $p_i(\beta_i)=\beta$,
and $p_i(\alpha_j)=p_i(\beta_j)=1$ ($j\neq i$). Using the same
letter $d$, we define $d:\pi \to \mathbb{Z}$ by $d(x)=\sum_{i=1}^g
d(p_i(x))$. For $\varphi \in MC(\Sigma)$, we see that the map $\pi
\to \mathbb{Z},\ x\mapsto d( \varphi^{-1}  (x))-d(x)$ is a
homomorphism. In this way, we get a map $\tilde{f}: MC(\Sigma)\to H$
as the composite
$$MC(\Sigma) \to {\rm Hom}(\pi,\mathbb{Z})={\rm Hom}(H,\mathbb{Z})\cong H,$$
where the last isomorphism is induced by homology intersection
numbers: $H\to {\rm Hom}(H,\mathbb{Z}),\ a\mapsto (y\mapsto (a\cdot
y))$.   Moreover, $\tilde{f}$ satisfies the cocycle condition:
$\tilde{f}(\varphi \psi)=\tilde{f}(\varphi)+\varphi \cdot
\tilde{f}(\psi)$ for any $\varphi,\psi\in MC(\Sigma)$.

Morita \cite{Morita89} proved that for $g\ge 2$,   the  group
$H^1(MC(\Sigma); H)$ is an infinite cyclic group generated by the
class $k=[\tilde{f}]$. Actually, Earle \cite{Ea} already found a
cocycle representing $k$, and following \cite{Kaw2}, we call $k$
{\it the Earle class}.

\begin{theorem}\label{thm:m}
We have
$$[\tilde{m}^a]=-2k \in H^1(MC(\Sigma); H).$$
\end{theorem}

  The minus sign in the formula comes from the difference between
the convention here and that of \cite{Morita89}, where a map $f:
MC(\Sigma)\to H$ is called a twisted 1-cocycle if
$f(\varphi\psi)=\psi^{-1}\cdot f(\varphi)+f(\psi)$. If $f$ is a
twisted cocycle in this sense, then the map $MC(\Sigma)\to H,
\varphi \mapsto f(\varphi^{-1})$ is a twisted 1-cocycle in our
sense.

To prove Theorem~\ref{thm:m},   we evaluate the two cocycles on a
{\it torus BP map}. Let $C_1$ and $C_2$ be disjoint simple closed
curves on $\Sigma$ such that their union is the boundary of a
torus-minus-two-disks embedded in $\Sigma$. The mapping class
$t_{C_1}t_{C_2}^{-1}$ is called a torus BP map (BP stands for
bounding pair), where $t_{C_i}$ is the right handed Dehn twist along
$C_i$. A torus BP map is an element of $MC_1(\Sigma)$.

\begin{figure}[!h]
\begin{center}
\epsffile{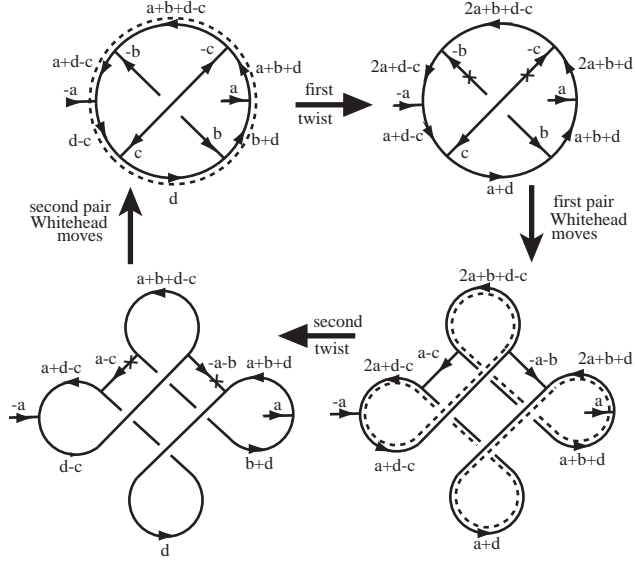}
\caption{BP torus map expressed as flips}
\label{fig:bp}
\end{center}
\end{figure}

\begin{lemma}
\label{lem:m(phi)} Let $\varphi$ be the torus BP map  illustrated in
Figure \ref{fig:bp} reproduced   from \cite{MP}.   Then
$\tilde{m}^a(\varphi)=4a$.
\end{lemma}

\begin{proof}
In \cite{MP}, Morita and Penner expressed a torus BP map by the
sequence of 14 flips in Figure \ref{fig:bp}, and we shall compute the
contributions from each step:

\noindent (1) the first Dehn twist (5 flips):
$$(a-c)+(a-b)+0+(a+c)+(a+b)=4a.$$
(2) the first pair of flips (2 flips):
$$(2a+b+d+a+d-c)+(a+d+2a+b+d-c)=6a+2b-2c+4d.$$
(3) the second Dehn twist (5 flips):
$$(-a+a+b)+(-a-a+c)+0+(-a-a-b)+(-a-c+a)=-4a.$$
(4) the second pair of flips (2 flips):
$$(-d-a-b-d+c)+(-b-d-a-d+c)=-2a-2b+2c-4d.$$
Taking the sum, we find $\tilde{m}^a(\varphi)=4a$.
\end{proof}

For definiteness, let us choose the particular fatgraph spine of
$\Sigma$ illustrated in Figure \ref{figtwo}, which also depicts two
simple closed curves $C_1,C_2$ and   part of a homology marking
viz.\ Figure \ref{figtwo}.  If $\varphi=t_{C_1}t_{C_2}^{-1}$, then
$\tilde{m}^a(\varphi)=4B_2$ in the notation of the standard basis
illustrated in Figure \ref{fig:basis}.

\begin{figure}[!h]
\begin{center}
\epsffile{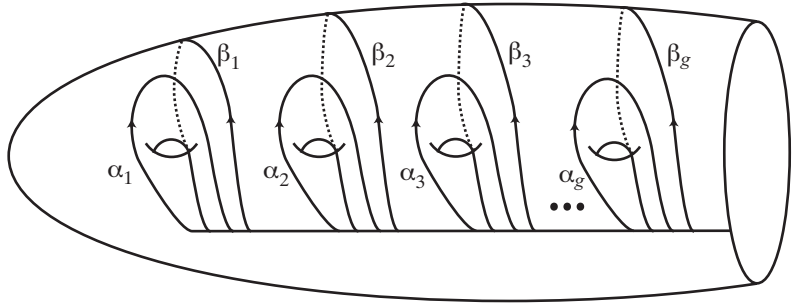}
\caption{Geometrically symplectic basis for $\pi$}
\label{fig:basis}
\end{center}
\end{figure}

\begin{figure}[!h]
\begin{center}
\epsffile{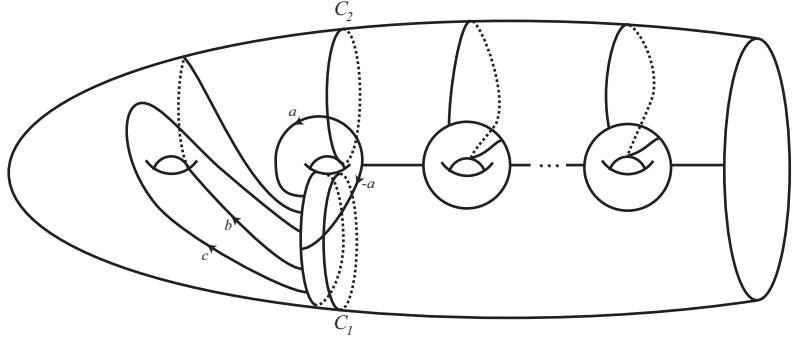}
\caption{Fatgraph spine, BP map $\varphi=t_{C_1}t_{C_2}^{-1}$ and
  a part of homological marking}
\label{figtwo}
\end{center}
\end{figure}

We next compute the value of $\tilde{f}$ on $\varphi=t_{C_1}t_{C_2}^{-1}$ in Figure \ref{figtwo}.
The action of $\varphi$ on $\pi$ is given by
$$\varphi(\alpha_1)=\gamma \alpha_1\gamma^{-1},\ \varphi(\beta_1)=\gamma \beta_1 \gamma^{-1},\
\varphi(\alpha_2)=\gamma \alpha_2 \beta_2,\ \varphi(\beta_2)=\beta_2,$$
and $\varphi(\alpha_i)=\alpha_i$, $\varphi(\beta_i)=\beta_i$ for
$i\ge 3$, where $\gamma
=\alpha_2\beta_2^{-1}\alpha_2^{-1}[\beta_1,\alpha_2]$.

\begin{lemma}
\label{lem:f(phi)} Let $\varphi=t_{C_1}t_{C_2}^{-1}$ be as above.
Then $\tilde{f}(\varphi)=-2B_2$.
\end{lemma}

\begin{proof}
The equality $\tilde{f}(\varphi)   =-\tilde{f}(\varphi^{-1}) =-2B_2$
is equivalent to $d(\varphi(x))-d(x)=0$ if $x\in \{ \alpha_i
\}_{i\neq 2}\cup \{ \beta_i \}_i$, and
$d(\varphi(\alpha_2))-d(\alpha_2)=-2$. Since
$d(\alpha_i)=d(\beta_i)=0$ for any $1\le i\le g$, it is sufficient
to show that $d(\varphi(\alpha_i))=0$ if $i\neq 2$,
$d(\varphi(\alpha_2))=-2$, and $d(\varphi(\beta_i))=0$ for $1\le
i\le g$. These equalities are verified by direct computations. For
example, to see $d(\varphi(\alpha_2))=-2$, note that
$\varphi(\alpha_2)=\alpha_2\beta_2^{-1}\alpha_2^{-1}[\beta_1,\alpha_1]\alpha_2\beta_2$
so $p_1(\varphi(\alpha_2))=\beta \alpha \beta^{-1} \alpha^{-1}
\stackrel{d}{\mapsto} -2$ and $p_2(\varphi(\alpha_2))=\alpha
\stackrel{d}{\mapsto} 0$.
\end{proof}

Theorem \ref{thm:m} follows directly from Lemmas \ref{lem:m(phi)} and \ref{lem:f(phi)},.

\begin{question}
Can one define a counterpart of the Earle class for a general
abelian $K$-marking ? What is the   group $H^1(MC_{\Box}(\Sigma);K)$
?
\end{question}

\bibliographystyle{amsplain}

\end{document}